\documentclass[12pt]{article}

\usepackage{amsfonts,amsmath,amsxtra,amsthm,amssymb,latexsym}
\usepackage{amssymb}  
\usepackage{bbm} 
\usepackage{graphicx}

 \usepackage[utf8]{inputenc}

\newtheorem{thm}{Theorem}[section]
 
 \newtheorem{lem}[thm]{Lemma}
 \newtheorem{prop}[thm]{Proposition}

 \theoremstyle{definition}
 \newtheorem{defn}{Definition}[section]
 \theoremstyle{remark}
 \newtheorem{rem}{Remark}[section]

 \numberwithin{equation}{section}

\DeclareMathOperator{\im}{Im}

\DeclareMathOperator{\Bd}{bd}

\def\RR{\mathbb R}
\def\CC{\mathbb C}
\def\NN{\mathbb N}
\def\ZZ{\mathbb Z}

\def\BB{\mathbb B}
\def\AAA{\mathbb{A}}

\def\ii{\mathrm{i}}

\def\wt{\widetilde}

\def\de{\delta}
\def\ep{\epsilon}

\def\Ga{\Gamma}

\def\De{\Delta}
\def\Si{\Sigma}
\def\si{\sigma}

\def\N{\mathcal{N}}
\def\W{W}

\def\V{V}

\def\G{\mathcal{G}}
\def\Gfr{\mathfrak{G}}

\def\cyc{\mathfrak{c}}

\def\id{\mathfrak{e}}

\begin{document}
\title{Generic asymptotics of resonance counting function for 
Schrödinger point interactions}
\author{}
\date{}
\maketitle

{\center 
{\large 
Sergio Albeverio$^{\text{\;a}}$ and Illya M. Karabash$^{\text{\;b,c,*}}$
\\[4ex]
}}

\begin{minipage}[t]{0.25\textwidth}
\quad
\end{minipage}
\begin{minipage}[t]{0.7\textwidth}
\it This work is dedicated to the dear Memory of Boris Pavlov. 
The first author had the great experience to meet him personally at 
a conference in Dubna back in 1987. At that time Boris was developing his original approach to 
point interactions and was the leader of a very strong group of young enthusiastic mathematicians
working in this area. From then on our steady friendship developed, with him and his coworkers.
The authors are very grateful to Boris for the many insights he has provided, 
that also influenced much of our work. We deeply mourn his departure.
\end{minipage}                                                                       

\vspace{4ex}

{\small \noindent
$^{\text{a}}$  Institute for Applied Mathematics, Rheinische Friedrich-Wilhelms Universität Bonn,
and Hausdorff Center for Mathematics, Endenicher Allee 60,
D-53115 Bonn, Germany\\[1mm]
$^{\text{b}}$
Mathematical Institute, Rheinische Friedrich-Wilhelms Universität Bonn,
 Endenicher Allee 60, D-53115 Bonn, Germany \\[1mm]
 $^{\text{c}}$ Institute of Applied Mathematics and Mechanics of NAS of Ukraine,
Dobrovolskogo st. 1, Slovyans'k 84100, Ukraine\\[1mm]
$^{\text{*}}$ Corresponding author: i.m.karabash@gmail.com\\[2mm]
E-mails: 
albeverio@iam.uni-bonn.de, i.m.karabash@gmail.com
}

\begin{abstract}
We study the leading coefficient in the asymptotic formula 
$ \N (R) = \frac{W}{\pi} R + O (1)  $, $R\to \infty$,  for the resonance counting function 
$\N (R)$ of Schrödinger Hamiltonians with point interactions. 
For such Hamiltonians, the Weyl-type and non-Weyl-type asymptotics of $\N (R)$ was introduced 
recently in a paper by J. Lipovský and V. Lotoreichik (2017).
In the present paper, we prove that the Weyl-type asymptotics is generic.
\end{abstract}

{
\small \noindent
MSC-classes: 
35J10, 
35B34, 
35P20, 
05C90, 
81Q37, 
81Q80 
\\[2mm]

\tableofcontents

}

\section{Introduction}
\label{s:intro}

The asymptotics as $R \to \infty$ of the counting function $\N_{H_{a,Y}} (R)$ for the resonances of 
a `one particle, finitely many centers' Schrödinger Hamiltonian $H_{a,Y}$ acting 
in the complex Lebesgue space $L^2 (\RR^3)$ and 
associated with the formal differential expression 
\begin{equation}\label{e:H}
-\De u (x) + `` \sum_{j=1}^N \mu (a_j) \de (x - y_j) u (x)  ``, \quad x=(x^1,x^2,x^3) \in \RR^3 , \ N \in \NN, 
\end{equation}
have being studied recently in \cite{LL17}, 
where the existence of Weyl-type and non-Weyl type 
asymptotics of $\N_{H_{a,Y}} (\cdot)$ have been proved and similarities 
with the case of quantum graphs \cite{DEL10,DP12} have been noticed.
The goal of the present paper is to show that 
the case of Weyl-type asymptotics of $\N_{H_{a,Y}} (\cdot)$ is generic for operators of the form 
(\ref{e:H}).

We denote by $\De$ the self-adjoint Laplacian in $L^2 (\RR^3)$ and assume throughout 
the paper that $N \ge 2$, where $N$ is the number of point interaction 
\emph{centers} $y_j \in \RR^3$, which are assumed to be distinct, i.e., $y_m \neq y_j$ if $m \neq j$.
The $N$-tuple of centers $(y_j)_{j=1}^N \subset (\RR^3)^N$ is denoted by $Y$.
The numbers $a_j \in \CC$ are the 'strength' parameters for the point interactions forming a tuple 
$a=(a_j)_{j=1}^N \in \CC^N$.

Roughly speaking, point interactions correspond to potentials expressed by the 
Dirac measures $\delta (\cdot - y_j)$ and play
the role of potentials in formula (\ref{e:H}) (this can be taken as definition in the 1-D case 
of Sturm-Liouville differential operators). Rigorously, in 3-D case, the point interaction Hamiltonian $H_{a,Y}$ associated with (\ref{e:H})
can be  introduced as a densely defined closed operator in the Hilbert space 
$L^2 (\RR^3)$ via a Krein-type  formula for 
the difference $(H_{a,Y} - z^2 )^{-1} - (-\De - z^2 )^{-1} $ of the perturbed and unperturbed resolvents of operators $H_{a,Y}$ and $-\De$, respectively. 
For the definition of $H_{a,Y}$ and for the meaning of the `strength' parameters and the factors $\mu (a_j) $ 
in (\ref{e:H}), we refer to 
\cite{AFH79,AGH82,AGHH12,AK00} in the case $a_j \in \RR$, and to 
$\cite{AGHS83,AK17}$ in the case $a_j \not \in \RR$ (see also Section \ref{s:Def}). 
Note that, in the case $(a_j)_{j=1}^N \subset \RR^N$, the operator $H_{a,Y} $ 
is self-adjoint in $L^2 (\RR^3)$; and 
in the case $(a_j)_{j=1}^N \subset (\overline{\CC}_-)^N $, 
$H_{a,Y}$ is closed and maximal dissipative 
(in the sense of \cite{E12}, or in the sense that $\ii H_{a,Y}$ is maximal accretive \cite{K13}).

Eigenvalues and (continuation) resonances $k$ of the corresponding operator $H_{a,Y}$ 
are connected with the special $N\times N$-matrix $\Ga (z)  $,
which is a function of  the spectral parameter $z$ and depends also on $Y$ and $a$. 
The matrix-function $\Ga (\cdot) $ appears naturally as a part of the expression for 
$(H_{a,Y} - z^2 )^{-1} - (-\De - z^2 )^{-1} $,
see \cite{AGHH12,AH84,AK17} and Section \ref{s:Def}.
The set $\Si (H_{a,Y})$ of resonances  associated with 
$H_{a,Y}$ is defined as the set of 
zeroes $k$ of the determinant 
$
\det \Ga (\cdot)  
$,
which is an analytic of $z$ function.

This definition follows the logic of \cite{DZ17,DEL10,LL17} and 
slightly differs from that of the original definition 
\cite{AGHH12,AH84} since it includes in the set of resonances 
the zeroes $k \in \CC_+ :=\{ z \in \CC : \im z>0 \}$, which correspond to eigenvalues $k^2$ of $H_Y$. 
It is easy to see  \cite{AGHH12,AH84,AK17} that $H_{a,Y}$ has only a finite number of eigenvalues and so 
the inclusion of corresponding $k \in \CC_+$ does not essentially 
influence the asymptotics for $R \to \infty$ 
of the counting function $\N_{H_{a,Y}} (\cdot)$, which is defined by 
\[
\N_{H_{a,Y}} (R) := \# \{ k \in \Si (H_{a,Y}) : |k|<R \} .
\]
Here $\#E$ is the number of elements of a multiset $E$.

When the number of resonances in a certain domain is counted,
$\Si (H_{a,Y})$ has to be understood as a multiset, i.e., an unordered set in which 
an element $e$ can be repeated a finite number $m_e \in \NN$ of times 
(this number $m_e$ is called the multiplicity of $e$). 
The multiplicity of a resonance $k$ is, by definition, its multiplicity as a 
zero of $\det \Ga  (\cdot)$,
and it is always finite since the resolvent set of $H_{a,Y}$ is nonempty (see \cite{AGHH12,AK17}).
The definition of the resonance counting function takes this multiplicity into account.

The study of the counting function $\N_{-\De+V} (\cdot)$ for scattering poles 
of Schrödinger Hamiltonians $-\De+V$ in 
$L^2 (\RR^n)$ 
with odd $n \ge 3$ was initiated in 
\cite{M83} (for the relation between the notions of scattering poles and resonances, 
see \cite{DZ17}). This study was continued and extended to obstacle and geometric scattering in a number of papers 
(see e.g. \cite{Z89DMJ,F98,CH05,CH08,DZ17,Z17} 
and references therein). 
In particular, it was proved in \cite{CH05,CH08} that for odd $n \ge 3$ the formula 
\begin{equation} \label{e:limsup}
\limsup_{R \to \infty} \frac{\log \N_{-\De+V} (R)}{\log R} = n 
\end{equation}
is generic for compactly supported $L^\infty$-potentials $V$. Generally, in such settings, 
only the bound $\limsup_{R \to \infty} \frac{\log \N_{-\De+V} (R)}{\log R} \le n$ is proved \cite{Z89DMJ}
(see the discussion of a related 
open problem in \cite{CH05,Z17}).

During the last two decades, wave equations, resonances, and related optimization problems 
on structures with combinatorial geometry and graph theory backgrounds have attracted a 
substantial attention, in particular, due to their engineering applications, 
see monographs \cite{AGHH12,AK00,BK13,P12}, papers 
\cite{AK17,DEL10,DP12,DK07,EFH07,HL17,Ka13_KN,Ka14} 
and references therein. 
One of the earliest studies of scattering on graphs was done by Gerasimenko and Pavlov \cite{GP88}.

In \cite{DP12,DEL10}, the asymptotics $\N_{\Gfr} (R) = \frac{2 \W_\Gfr}{\pi} R + O (1)$ as $R\to\infty$ 
for the resonance counting function of a non-compact quantum graph $\Gfr$ have been obtained
and it was shown that the nonnegative constant $\W_\Gfr$, which was called 
\emph{the effective size of the graph}, is less or equal to the sum of lengths of 
the internal edges of the graph.
It was said that the quantum graph has a Weyl resonance asymptotics
if $\W_\Gfr$ equals the sum of lengths of internal edges.
Special attention was paid in \cite{DEL10,DP12} 
to the cases where non-Weyl  asymptotics holds, i.e., to the cases where 
$\W_\Gfr$ is strictly less than the sum of lengths of internal edges.

In the recent paper \cite{LL17}, it was noticed  that the resonance theories for point interactions and for 
quantum graphs have a lot in common, and
the asymptotics 
\begin{gather} \label{e:NGas}
\N_{H_{a,Y}} (R) = \frac{\W (H_{a,Y})}{\pi} R + O (1)  \text{ as $R\to \infty$}
\end{gather}
was established for point interaction Hamiltonians $H_{a,Y}$ 
with a certain positive constant $\W (H_{a,Y})$, which is called \emph{an effective size of the set} $Y$.
(Note that (\ref{e:NGas}) holds 
for the case $N\ge 2$; in the simple case $N=1$ it is obvious that only one resonance exists.) 
On the other hand, \emph{the size of the family of centers}  Y was defined by  
\[
\V (Y) :=  \max_{\si \in S_N} \sum_{j=1}^N | y_{j} -  y_{\si (j)}| ,
\]
where the maximum is taken over all permutations $\si $ in the symmetric group $S_N$.
The asymptotics of $\N_{H_{a,Y}} (R)$ for $R \to \infty$ was called  of \emph{Weyl-type if 
the effective size $\W (H_{a,Y})$ in (\ref{e:NGas}) coincides with the size $\V (Y)$.} An example of $H_{a,Y}$ with 
non-Weyl-type asymptotics was constructed in \cite{LL17}. (Note that, 
while \cite{LL17} considers only the case where all $a_j$ coincide and are 
real, the results and proofs of \cite{LL17} can be extended to the case of arbitrary 
$a \in \CC^N$ almost without any changes.)

The present paper studies how often the equality $\W (H_{a,Y})=\V (Y)$, i.e., 
Weyl-type 
asymptotics, happens.
To parametrize rigorously the family of Hamiltonians $H_{a,Y}$,
let us consider $Y$ as a vector in the space $(\RR^3)^N$ of ordered $N$-tuples $y=(y_j)_{j=1}^{N}$
with the entries $y_j \in \RR^3$. We consider $(\RR^3)^N$ as a linear normed space 
with the $\ell^2$-norm $| y |_2 = (\sum |y_j|^2)^{1/2}$.
Then the ordered collection $Y$ of  centers is identified with 
an element of the subset  $\AAA \subset (\RR^3)^N$ defined by 
\[
\AAA := \{ y \in (\RR^3)^N \ : \ y_j \neq y_{j'} \ \text{ for } \ j \neq j' \} .
\]
We consider $\AAA$ as a metric space with the distance function induced by the norm $| \cdot |_2$.

The following theorem is the main result of the present paper.
It shows that Weyl-type asymptotics is generic for point interaction Hamiltonians and gives 
a precise sense to this statement.
  
\begin{thm} \label{t:main}
There exists a subset $\AAA_1 \subset \AAA$ that is open and dense in the metric space $\AAA$ and has 
also the property that, for every $Y \in \AAA_1$ and every $a \in \CC^N$, 
the counting function for the resonances of $H_{a,Y}$ has the Weyl-type asymptotics
$\N_{H_{a,Y}} (R) = \frac{\V (Y)}{\pi} R + O (1)$  as $R\to \infty$.
\end{thm}

The proof is constructive and is given in Section \ref{ss:ProofMain}.

\textbf{Notation}. 
The following standard sets are used: the lower 
and upper 
complex half-planes 
$\CC_\pm = \{ z \in \CC : \pm \im z > 0  \}$,
the set $\ZZ$ of integers, 
the closure  $\bar{S}$ of a subset of a normed space $U$, in particular,
$\overline{\CC}_\pm = \{ z \in \CC : \pm \im z \ge 0  \}$, 
open balls 
$
\BB_\epsilon (u_0) =\BB_\epsilon (u_0; U) 
:= \{u \in U \, : \, \rho_U (u, u_0) < \epsilon \}
$
in a metric space $U$ with the distance function $\rho_U (\cdot,\cdot)$ (or in a normed space).

By $y_j \sim y_m$ we denote an edge between vertices $y_j$ and $y_m$ in an undirected 
graph $\G$ (so $y_j \sim y_m$ and $y_m \sim y_j$ is the same edge).
Directed edges in a directed graph $\overrightarrow{\G}$ will be called bonds 
in accordance with \cite{BK13} and denoted by $y_j \leadsto y_m$, which means that the bond is from 
$y_j $ to $y_m$ (note that this notation is slightly different from that of \cite{BK13}).

\section{Resonances as zeroes of a characteristic determinant}
\label{s:Def}

Let the set $Y=\{ y_j \}_{j =1}^N$ consist of $N\ge 2$ distinct points 
$y_1$, \dots, $y_N$ in $\RR^3$. Let $a = (a_j)_{j=1}^N \in \CC^N$ be the $N$-tuple  of 
the strength parameters.
The operator $H$ associated with (\ref{e:H}) is defined in \cite{AFH79,AGH82,AGHH12} for the case 
of real $a_j$, and in \cite{AK17} for $a_j \in \CC$.
It is a closed operator in the complex Hilbert space $L^2 (\RR^3)$ 
 and it has a nonempty resolvent set.
The spectrum of $H$ consists of the essential spectrum $[0,+\infty)$ and an 
at most finite set 
of points outside of $[0,+\infty)$ \cite{AK17} (all of those points are eigenvalues).

The resolvent $(H-z^2)^{-1}$ of $H$ is defined in the classical sense 
on the set of $ z \in \CC_+ $ such that $z^2$ is not in the spectrum, 
and has the integral kernel 
\begin{gather} \label{e:Res}
(H-z^2)^{-1} (x,x')  = G_z (x-x') + \sum_{j,j' = 1}^N G_z (x-y_j)
\left[ \Ga_{a,Y} \right]_{j,j'}^{-1} G_z (x' - y_{j'} ) , 
\end{gather}
where $x,x' \in \RR^3 \setminus Y$ and $x \neq x'  $, see e.g. \cite{AGHH12,AK17}.
Here 
\[
G_z (x-x') := \frac{e^{\ii z |x-x'|}}{4 \pi |x-x'|}
\] 
is the integral kernel associated with
the resolvent $(-\De - z^2)^{-1}$ of 
the kinetic energy Hamiltonian $-\De$;
$\left[ \Ga_{a,Y} \right]_{j,j'}^{-1}$ denotes the $j,j'$-element of the inverse to 
the matrix 
\begin{gather} \label{e:Ga}
\Ga_{a,Y} (z) = \left[ \left( a_j - \tfrac{\ii z}{4 \pi} \right) \de_{jj'} 
- \wt G_z  (y_j-y_{j'})\right]_{j,j'=1}^{N}, \text{ where }
\wt G_z (x) := \left\{ \begin{array}{rr} G_z (x), & 
x \neq 0 \\
0 , & 
x = 0  \end{array} \right.  .
\end{gather}


The multi-set of (continuation) resonances $\Si (H)$ associated with the operator $H$ 
(in short, resonances of $H$) is by definition the set of zeroes of 
the determinant 
$
\det \Ga_{a,Y} (\cdot)  ,
$
which we will call the characteristic determinant.
This definition follows \cite{DZ17} and slightly differs 
from the one used in \cite{AH84,AGHH12}
because isolated eigenvalues are now 
also included into $\Si (H)$.
For the origin of this and related approaches to the understanding of resonances, we refer
to \cite{AH84,DZ17,RSIV78,V72} and the literature therein. 
The multiplicity of a resonance $k$ will be understood as the multiplicity 
of a corresponding zero of the analytic function 
$\det \Ga_{a,Y} (\cdot)$ (see \cite{AGHH12}).

 An important common feature of quantum graphs and point interaction Hamiltonians is that 
the function that is used to determine resonances as its zeroes is an exponential polynomial. That is, 
the function 
\begin{gather} \label{e:D}
D (z) := (- 4 \pi)^N \det \Ga_{a,Y} (z)
\end{gather}
has the form
\begin{equation} \label{e:CanForm}
 \sum_{j=0}^\nu P_{b_j} (z) e^{\ii b_j z} ,
\end{equation}
where $\nu \in \NN \cup \{0\}$, $b_j \in \CC$, and $P_{b_j} (\cdot)$ are polynomials. 
One can see that, for the particular case of $D(\cdot)$,
$b_j $ are real and nonnegative numbers.

In what follows, we assume that the polynomials 
$P_{b_j} (\cdot)$ in the representation (\ref{e:CanForm}) for $D (\cdot)$
are nontrivial in the sense that $P_{b_j} (\cdot) \not \equiv 0$,
and that all $b_j$ in (\ref{e:CanForm})
are distinct and ordered such that  
\[
b_0 < b_1 < \dots < b_\nu .
\] 
Under these assumptions, the set 
$\{ b_j \}_{j=0}^\nu$ and the representation  
(\ref{e:CanForm}) are unique.

It is obvious that $b_0 = 0$ and the corresponding polynomial is equal to
\[
P_0 (z)  = \prod_{j=1}^N  (\ii z - 4 \pi a_j ) .
\] 
It is easy to notice \cite{AK17} that $\nu \ge 1$, i.e., there are at least two summands in the sum 
(\ref{e:CanForm}) and at least one of them involves nonzero number $b_j $.
(Let us recall that $N \ge 2$ is assumed throughout the paper).

The asymptotic behavior of the counting function $\N_{H}$ for resonances of $H$ is given by the formula
\begin{equation} \label{e:Nas}
\N_H (R) = \frac{b_\nu}{\pi} R + O(1) \text{ as $R\to \infty$},
\end{equation}
which was derived in \cite{LL17} from \cite[Theorem 3.1]{DEL10} (for more general versions of this result 
in the context of the general theory of exponential polynomials see \cite{BG12} and references therein).
So, following the terminology of \cite{LL17}, 
\begin{equation} \label{e:bnu=W}
\text{$b_\nu$ is the effective size $W (H)$}
\end{equation}
associated with the N-tuples $Y$ and $a$ that define the Hamiltonian $H$.
 
 \begin{rem} 
 This raises  the natural question of whether there exist a 
 family of centers $Y$ such that 
 the effective size $W (H_{a,Y})$ of $H_{a,Y}$ might change with the change of 
 the `strength' tuple $a \in \CC^N$. 
 (Note that we do not allow  $a_j$ to take the value $\infty$. Otherwise, the answer becomes obvious 
since $a_j = \infty$ means that the center $y_j$ is excluded from $Y$, see \cite{AGHH12,AK17}).
 \end{rem}


\section{Absence of cancellations in Leibniz formula}

In this section, we assume that the tuple $a$ is fixed
and consider the operator $H$ and the set $\Si (H)$ of its resonances as functions 
of the family $Y$ of interaction centers. Therefore we will use the notation $H_Y$ for $H$,
and 
\begin{align} \label{e:DY}
D_Y (z) := (-4 \pi)^N \det \Ga_{a,Y} (z)
\end{align} 
for the corresponding modified version (\ref{e:D}) of 
the characteristic determinant (\ref{e:Ga}).

The Leibniz formula expands $D_Y (z)$ into the sum of terms
\begin{equation} \label{e:Dterms}
e^{\ii z \V_\si (Y)} P^{[\si,Y]} (z)
\end{equation}
taken over all permutations $\si$ in the symmetric group $S_N$,
where the constants $\V_\si (Y) \ge 0$  depends on $\si$ and $Y$, and 
$P^{[\si,Y]} (\cdot) $ are polynomials in $z$ depending on $\si$ and $Y$. They have the form 
\begin{gather} \label{e:VsiY}
\V_\si (Y) := \sum_{j=1}^N |y_j - y_{\si (j)}| = \sum_{j : \si (j) \neq j  } |y_j - y_{\si (j)}|, \\
P^{[\si,Y]} (z) := (-1)^{\ep_\si} K_1 (\si,Y) \prod_{j : \si (j) = j  } (\ii z -4\pi a_j) ,
\label{e:p si}
\end{gather}
where $K_1$ is the positive constant depending on $\si$ and $Y$,
\[
K_1 (\si,Y):= \prod_{j : \si (j) \neq j  } |y_j - y_{\si (j)}|^{-1} 
\]
($K_1 (\id,Y) := 1$ in the case where $\si$ is the identity permutation 
$ \id = [1] [2] \dots [N]$) and 
$\ep_\si$ is the permutation sign (the Levi-Civita symbol).

Here and below we use the square brackets notation of the textbook \cite{Lang02}
for permutation cycles, omitting sometimes, when it is convenient, 
the degenerate cycles consisting of one element. 
For each permutation $\si$, there exists a decomposition 
\begin{gather} \label{e:si=cyc}
\si = \prod_{m=1}^{M(\si)} \cyc_m 
\end{gather}
of $\si$ into disjoint cycles $\cyc_m$ (in short, the \emph{cycle decomposition} of $\si$). 
The changes in the order of cycles in the product 
(\ref{e:si=cyc}) does not influence the result of the product. Up to such variations of order,
the decomposition (\ref{e:si=cyc}) is unique (see e.g. \cite{Lang02}). 
So the number $M(\si)$ of cycles in (\ref{e:si=cyc}) is a well defined 
function of $\si$. It is connected with the permutation sign $ \ep_\si$ by the well-known equality
 \begin{equation} \label{e:ep=cyc}
 \ep_\si  = (-1)^{N-M(\si)} .
 \end{equation}
(To see this it is enough to conclude from the equality 
$
[1 \  2 \dots  \ n] = [1 \ 2] [2 \ 3] \dots [(n-1) \  n]
$
  that the sign of a cycle $\cyc_m$ equals to 
 $(-1)^{\# (\cyc_m) -1}$, 
 where $\# (\cyc_m)$ is the number of elements involved in the cycle $\cyc_m$.)

We will use some basic notions of graph theory that are concerned with
the directed and undirected graphs with lengths. Such graphs can be realized as metric graphs
or as weighted discrete graphs. Because of connections of the topic of this paper 
with quantum graphs (see \cite{LL17}),
we try to adapt terminology and notation close to (but not coinciding with) that of 
 the monographs \cite{BK13,P12}.

The numbers $\V_\si (Y)$ have a natural geometric description from 
the point of view of pseudo-orbits of 
directed metric graph  having 
vertices at the centers $y_j$, $j=1$, \dots, $N$ (see \cite{LL17} and references therein). 
Namely, $\V_\si (Y)$
is the \emph{metric length} (the sum of length of bonds) 
of the \emph{directed graph} $\overrightarrow{\G_\si}$ associated with $\si \in S_N$ 
and consisting, by definition,
of \emph{bonds} $y_j \leadsto y_{\si (j)}$,
$j=1$, \dots, $N$, of metric length $|y_j - y_{\si (j)}|$. 
Note that loops of zero length from a vertex to itself are allowed in $\overrightarrow{\G_\si}$.
Namely, in the case
where the cycle  decomposition of a permutation $\si$ includes a degenerate cycle $[j]$ (i.e., $j=\si (j)$),
the bond $y_j \leadsto y_{\si (j)}$ degenerates into the loop $y_j \leadsto y_j$ 
from $y_j$ to itself, 
which has zero length. If the cycle decomposition of $\si$ contains the cycle $[j \ \si(j)]$, 
then the corresponding cycle of the directed graph $\overrightarrow{\G_\si}$ consists of the two bonds 
$y_j \leadsto y_{\si (j)}$, $y_{\si(j)} \leadsto y_j$  between 
$j$ and $\si(j)$ with two opposite 
directions, and so, the contribution of this cycle to the metric length of 
$\overrightarrow{\G_\si}$ 
is $2 |y_j - y_{\si (j)}|$.

This, in particular, explains why the number
\begin{equation} \label{e:V(Y)=max}
\V (Y) = \max_{\si \in S_N} \V_\si (Y) 
\end{equation}
is called in \cite{LL17} the size of $Y$.

The coefficients $b_j$ in (\ref{e:CanForm}) are called \emph{frequencies} 
of the corresponding exponential polynomial $D_Y (\cdot)$.
Similarly, $V_\si (Y)$ is the frequency of the term (\ref{e:Dterms}) in the Leibniz formula.
By (\ref{e:V(Y)=max}), there exists a term of the form (\ref{e:Dterms}) that has $\V (Y)$ as its frequency.
In the process of summation of the terms (\ref{e:Dterms}) in the Leibniz formula some of the 
terms  may cancel so that, 
for a certain permutation $\si \in S_N$, $V_\si (Y)$ is not a frequency of $D_Y (\cdot)$.
If this is the case, we say that there is \emph{frequency cancellation} for the frequency $V_\si (Y)$.
An example of frequency cancellation for the highest possible frequency $\V (Y)$ 
have been constructed in \cite{LL17} to prove that non-Weyl asymptotics 
is possible for $H_{a,Y}$.

By $\G_\si$ we denote the metric pseudograph (i.e., 
the undirected metric graph with possible degenerate loops and 
multiple edges) that is produced from the directed 
graph $\overrightarrow{\G_\si}$ by stripping off the direction for all bonds. So if $\si(j) = j$, 
$\G_\si$ contains the loop-edge $y_j \sim y_j$ of zero length. If $\si(j) \neq  j$ and $\si(\si(j))=j$, 
$\G_\si$ contains two identical edges $y_j \sim y_{\si(j)}$ each of 
them contributing to the metrical length $\V_\si (Y)$ of $\G_\si$ (that is the multiplicity of 
the edge $y_j \sim y_{\si(j)}$ is 2). These two cases describe all `nonstandard`
situations where the multiplicity of an edge is strictly larger than $1$, or 
a loop can appear in $\G_\si$. That is, if $\si(j) \neq j \neq \si(\si(j))$, then $\G_\si$ 
has exactly two edges involving $y_j$, namely, $y_j \sim y_{\si^{\pm1} (j)}$, 
which are distinct and of multiplicity 1.

\begin{defn} \label{d:BondEq}
We will say that two permutations $\si$ and $\si'$ are \emph{edge-equivalent} and write 
$\si \cong \si'$
if $\G_{\si} = \G_{\si'}$,.
\end{defn} 

Here the equality $\G_{\si} = \G_{\si'}$ is understood in the following sense:
for any $j,j' \in [1,N] \cap \NN$,
\begin{gather} \label{e:G=G}
\text{the multiplicities of the edge
$y_j \sim y_{j'}$ in the graphs $\G_\si $ and $\G_{\si'}$ coincide.}
\end{gather}

It is easy to see that $\si \cong  \si'$ exactly when the cycle decomposition 
of $\si'$ can be obtained from that of $\si$ by inversion of some of the cycles, i.e.,
for $\si$ with the cycle decomposition (\ref{e:si=cyc}), 
the edge-equivalence class of $\si$ consists of permutations of the form 
$\prod_{m=1}^{M(\si)} \cyc_m^{\alpha_m} $, where each $\alpha_m$ takes either the value $1$,
or $(-1)$.

From (\ref{e:ep=cyc}) and (\ref{e:VsiY}) we see that, 
\begin{gather} \label{e:ep=epV=V}
\text{if $ \si \cong \si'$, then  $\ep_\si  = \ep_{\si'} $ and $V_\si (Y) = V_{\si'} (Y)$.}
\end{gather}

The reason for the introduction of the edge-equivalence is the following statement.

\begin{prop} \label{p:WeylAsV<>V}
Assume that $Y$ is such that the following assumption hold:
\begin{gather} \label{a:V}
\text{$V_\si (Y) = V_{\si'} (Y)$ only if $\si \cong \si'$}.
\end{gather}
Then the Weyl-type asymptotics takes place, i.e.,  $\W (H_{a,Y}) = \V (Y)$.
\end{prop}

\begin{proof}
Under condition (\ref{a:V}), it is clear from (\ref{e:ep=epV=V}) and the form (\ref{e:p si}) 
of the polynomials $P^{[\si,Y]} (\cdot)$, that in the process of summation of (\ref{e:Dterms}) 
by the Leibniz formula the terms 
with the frequency $\V (Y)$ cannot cancel each other in $D_Y (\cdot)$. Thus, $\V (Y) = b_\nu = \W (H_{a,Y}) $ 
(see (\ref{e:Nas}), (\ref{e:bnu=W})).
\end{proof}

\begin{thm} \label{t:V=V}
Let $\si, \si' \in S_N$. Then the following statements are equivalent:
\item[(i)] $\si \cong \si'$;
\item[(ii)] $\V_\si (Y) = \V_{\si'} (Y)$ for all $Y \in \AAA$;
\item[(iii)] $\V_\si (Y) = \V_{\si'} (Y)$ for all $Y$ in a certain open ball $\BB_{\de_0} (Y_0)$ 
of the metric space $\AAA$,
where $\de_0>0$ and $Y_0 \in \AAA$.
\end{thm}

\subsection{Proof of Theorem \ref{t:V=V}}

The implications (i) $\implies$ (ii)  $\implies$ (iii) are obvious.

\emph{Let us now prove (iii) $\implies$ (ii).} 

\begin{lem}
Let the closed segment $[Y_0,Y_1] = \{ Y(t)=(1-t) Y_0 + t Y_1 \ : \ t \in [0,1]\}$ belong to the set 
$\AAA$ and $Y_0 \neq Y_1 $.
Let $Y_0$ satisfy statement (iii) of Theorem \ref{t:V=V}. Then:
\item[(1)]  $\V_\si (Y) = \V_{\si'} (Y)$ for all $Y \in [Y_0,Y_1]$;
\item[(2)] $Y_1$ satisfies statement (iii) of Theorem \ref{t:V=V} in the sense that 
$\V_\si (Y) = \V_{\si'} (Y)$ for all $Y \in \BB_{\de_1} (Y_1)$ with a certain $\de_1>0$.
\item[(3)] statement (ii) of Theorem \ref{t:V=V} is satisfied, i.e., $\V_\si (Y) = \V_{\si'} (Y)$ 
for all $Y \in \AAA$.
\end{lem}

\begin{proof}
Let $\de_2 := \min \{ | y- Y(t) |_2 \ : \ y \in \Bd \AAA, \  t \in [0,1] \} $, where 
\[
\Bd \AAA := \{ y \in (\RR^3)^N \ : \ y_{j} = y_{j'} \ \text{ for a certain } \ j \neq j' \}
\]
 is the boundary of the set $\AAA$ in the normed space  $(\RR^3)^N$. 
 Since $[Y_0,Y_1] \subset \AAA$ and $\AAA$ is open in $(\RR^3)^N$,
 we see that $\de_2>0$.
 
\emph{(1)} Consider the function $f(t) = \V_\si (Y(t)) - \V_{\si'} (Y(t))$ for 
$t \in (-\de_3,1+\de_3)$, 
where $ \de_3 := \de_2 / |Y_1 - Y_0|_2$. 

It follows from (\ref{e:VsiY}) that $f(\cdot)$  is analytic in 
the interval $(-\de_3,1+\de_3)$. Indeed, 
$|y_j  - y_{j'} |= \left(\sum_{m=1}^3 (y_{j,m} - y_{j',m})^2 \right)^{1/2}$, where $y_{j,m}$ and $y_{j',m}$, $m=1,2,3$,
are the $\RR^3$-coordinates of $y_j$ and $y_{j'}$, respectively.
Since $Y(t) \in \AAA$ for this range of $t$, the sum cannot be $0$ for $j \neq j'$, and 
we see that, for $(y_j (t))_{j=1}^N = Y(t)$, each $|y_j (t) - y_{j'} (t)|$ 
is a composition of  functions which are analytic in $t$.

Hence, we can consider a complex $t$ and extend $f(\cdot)$ as an analytic function in a neighborhood of 
$(-\de_3,1+\de_3)$ in the complex plane $\CC$.  
Since $Y_0$ satisfies (iii), we have $f(t) = 0 $ in a neighborhood of $0$. 
Due to analyticity, $f(t) = 0 $ for all $t \in (-\de_3,1+\de_3)$.
This proves (1).

\emph{(2)} It follows from the definition of $\de_2$ that 
$[Y_0 , Y] \subset \AAA$ for every $Y \in \BB_{\de_2} (Y_1) $. 
So statement  (1) of the lemma can be applied to each of these segments. This gives (2).

\emph{(3)} follows from statement (2) and the fact that $\AAA$ is 
piecewise linear path connected in 
$(\RR^3)^N$.
\end{proof}

\emph{Let us prove (ii) $\implies$ (i).} Assume that $\si \not \cong \si'$. 
Then, by Definition \ref{d:BondEq} and (\ref{e:G=G}), there exists a center
$y_{j_*}$ such that the sets of centers which are connected with $y_{j_*}$ 
in the graphs $\G_\si$ and $\G_{\si'}$ do not coincide.
This can happen in several situations, which, by a possible exchange of roles between $\si$ and $\si'$, 
can be reduced to the 
following six cases:
\begin{itemize}
\item[(a)] $\si (j_*) = j_*$ and $ \si' (j_*) \neq j_*$.
\item[(b)] Both graphs $\G_\si$ and $\G_{\si'}$ have the common nondegenerate edge $y_{j_*} \sim y_{j_1} $ of multiplicity one, 
but the 
second edge involving $y_{j_*}$ in $\G_\si$ and $\G_{\si'}$ do not coincide. 
To be specific let us assume that $\G_{\si}$ and $\G_{\si'}$ have the edges 
$ y_{j_*} \sim y_j $ and 
$y_{j_*} \sim y_{j'}$, resp., and that $y_{j_*}$, $ y_{j_1}$, $y_j$, $y_{j'}$ are distinct centers.
\item[(c)] The graph $\G_\si$ has two edges $y_{j_*} \sim y_{j_m} $, $m=1,2$, 
the graph $\G_{\si'}$ has two edges $ y_{j_*} \sim y_{j'_m} $, $m=1,2$, and 
all the five centers $y_{j_*}$, $ y_{j_m}$, $y_{j'_m}$, $m=1,2$, are distinct.
\item[(d)] The graph $\G_\si$ has the edge $y_{j_*} \sim y_{j} $ of multiplicity 2,
the graph $\G_{\si'}$ has two edges $ y_{j_*} \sim y_{j'_m} $, $m=1,2$,
and the 4 centers $y_{j_*}$, $ y_{j}$, $y_{j'_1}$, $y_{j'_2}$ are distinct.
\item[(e)] The graph $\G_\si$ has the edge $y_{j_*} \sim y_{j_1} $ of multiplicity 2,
the graph $\G_{\si'}$ has two edges $ y_{j_*} \sim y_{j_1} $,
$ y_{j_*} \sim y_{j_2} $, 
and the 3 centers $y_{j_*}$, $ y_{j_1}$, $y_{j_2}$ are distinct.
\item[(f)] The graph $\G_\si$ has the edge $y_{j_*} \sim y_{j} $ of multiplicity 2,
the graph $\G_{\si'}$ has the edge $ y_{j_*} \sim y_{j'} $ of multiplicity 2,
and the 3 centers $y_{j_*}$, $ y_{j}$, $y_{j'}$ are distinct.
\end{itemize}
Let us show that, for each of the 6 above situations,  statement (ii) of Theorem \ref{t:V=V} does not hold true under the assumption 
that $\si \not \cong \si'$.

\emph{Case (a).} The function $\V_\si (Y)$ is obviously constant for all small changes of $y_{j_*}$ since this center
is connected only with itself in $\G_\si$. This is not true for $\V_{\si'} (Y)$ because in 
$\G_{\si'}$, $y_{j_*}$
is connected with at least one of the other centers. Hence, statement (ii) of Theorem \ref{t:V=V}
does not hold true.

\emph{Case (b).} Let us take for $ t \in (0,1) $, $Y(t) \in \AAA$ such that 
$y_{j_*} = (1-t) y_j + t y_{j'}$,
but all the centers except $y_{j_*}$ do not depend on $t$. 
Then as $t$ changes in $(0,1)$, the function $\V_\si (Y(t)) - \V_{\si'} (Y(t))$ is strictly increasing.
This contradicts  statement (ii) of Theorem \ref{t:V=V}.

\emph{Cases (c)-(f) can be treated by arguments similar to that 
of Case (b) with some modifications}, which we consider briefly below.

\emph{Case (c).} One can take $Y(t)$, $t\in (-1,1)$, such that for $m=1,2$, 
the $\RR^3$-coordinates of $y_{j_m}$ are $(-m,0,0)$, the $\RR^3$-coordinates of 
$y_{j'_m}$ are $(m,0,0)$, and put
$y_{j_*} = (t,0,0)$. Thus, $\V_\si (Y(t)) - \V_{\si'} (Y(t))$ is strictly increasing for $t\in (-1,1)$
and so the statement (ii) of Theorem \ref{t:V=V} does not hold.

\emph{Case (d).} In the graph $\G_\si$, the center $y_{j_*}$ is connected only with
$y_j$. That is why it is easy to construct the evolution $Y(t)$ of $Y$ in such a way 
that only $y_{j_*}$ moves, the distance $|y_{j_*} - y_{j}|$ is constant, but 
$\sum_{m=1,2} |y_{j_*} - y_{j'_m}|$ changes, 
and so also $\V_{\si'} (Y(t))$ changes contradicting 
the statement (ii) of Theorem \ref{t:V=V}. For example, let $y_{j'_m} = ((-1)^m,0,0)$
for $m=1,2$, $y_{j} = (0,1,0)$, and assume that $y_{j_*} (t)$ move along a circle of radius 1 in $Ox^1x^2$-plane.

\emph{Case (e).} It is enough to choose 
$y_{j_*} (t) = (1-t) y_{j_1} + t y_{j_2}$ for $t \in (0,1)$ 
with all other centers fixed, and then to follow arguments of Case (b).

\emph{Case (f).} It is enough to put  $y_{j_*} (t) = (1-t) y_{j} + t y_{j'}$ and use 
the arguments of Case (b).
This completes the proof of Theorem \ref{t:V=V}.

\subsection{Proof of Theorem \ref{t:main}} 
\label{ss:ProofMain}

Let us denote by $n \in \NN$ the number of 
edge-equivalence classes in $S_N$ and let us take one representative $\wt \si_j$, $j=1,\dots, n$, in each of them.
Let 
\[
\AAA_1 := \{ Y \in \AAA \ : \ \V_{\wt \si_j} (Y) \neq \V_{\wt \si_m} (Y) \text{  if } j \neq m \} .
\]

\begin{lem} \label{l:A1}
The set $\AAA_1$ is open and dense in the metric space $\AAA$.
\end{lem}

\begin{proof}
Consider the sets 
\[
\AAA^{j,m} := \{ Y \in \AAA \ : \ \V_{\wt \si_j} (Y) \neq \V_{\wt \si_m} (Y) \} ,  
\quad j , m = 1,\dots, n.
\]
Since the functions $\V_{\wt \si_j} (\cdot)$ are continuous in $\AAA$, 
the sets $\AAA^{j,m}$ are open.

Let us show that $\AAA^{j,m}$ is dense in $\AAA$ whenever $j \neq m$. 
Assume ad absurdum that the converse is true. Then statement (iii)
of Theorem \ref{t:V=V} holds for $\wt \si_j$ and $\wt \si_m$. By Theorem \ref{t:V=V}, $\wt \si_j \cong \wt \si_m$.
This contradicts the choice of $\wt \si_j$ and $\wt \si_m$ as representative of different 
edge-equivalence classes.

We see that $\AAA_1 = \bigcap_{1 \le j < m \le n} \AAA^{j,m}$ is the intersection of a finite number 
of open dense sets. This completes the proof.
\end{proof}

Proposition \ref{p:WeylAsV<>V} shows that for each $Y \in \AAA_1$ and each $a \in \CC^N$ the Weyl-type 
asymptotics of $\N_{H_{a,Y}} (\cdot)$ takes place. This completes the proof of Theorem \ref{t:main}.

\vspace{3ex}
\noindent
\textbf{Acknowledgments.}\\ 
The authors are thankful to the anonymous referee for careful reading of the paper 
and valuable remarks that have improved the clarity of the proofs.
The authors are grateful to Pavel Exner and Vladimir Lotoreichik 
for stimulating discussions on resonances of  
quantum graphs and point interactions.
The second named author (IK) is grateful to 
Konstantin Pankrashkin, Olaf Post,  and Ralf Rueckriemen
for the invitation to the 3rd French-German meeting 
"Asymptotic analysis and spectral theory" (Trier, 25-29/09/2017), 
where some of the above discussions took place. 
The authors are grateful  to Gianfausto Dell'Antonio and Alessandro Michelangeli for the hospitality 
at SISSA during the workshop ``Trails in Quantum Mechanics and Surroundings''
in honour of the 85th birthday of Gianfausto Dell'Antonio
(Trieste, 29-30/01/2018) and at INdAM during the 3rd workshop ``Mathematical Challenges of Zero-Range Physics''
(Rome, 9-13/07/2018).
IK is grateful to Herbert Koch for discussions of geometric resonances 
and for the hospitality at the University of Bonn, to Jürgen Prestin for 
the hospitality at the University of Lübeck, and to Dirk Langemann for the hospitality at 
TU Braunschweig.

During various parts of this research, IK was supported by the Alexander von Humboldt 
Foundation, by VolkswagenStiftung project “Modeling, Analysis, and Approximation 
Theory toward applications in tomography and inverse problems”,
and by the  WTZ grant 100320049 ''Mathematical Models for Bio-Medical Problems'' 
jointly sponsored by BMBF (Germany) and MES (Ukraine).

\end{document}